\documentclass[11pt,a4paper]{elsarticle}
\usepackage{amsmath,amsthm,amsfonts,amssymb}
\usepackage{latexsym,euscript,dsfont}
\usepackage{xcolor}
\usepackage{txfonts}
\usepackage{mathrsfs}
\usepackage{bbm}
\usepackage{colortbl}
\usepackage{epstopdf}
\usepackage{graphicx}
\usepackage[style=base]{caption}
\usepackage{subcaption}
\usepackage{enumerate}
\usepackage{url}
\usepackage[pagewise]{lineno}

\theoremstyle{plain}
\newtheorem{Thm}{Theorem}[section]
\newtheorem{Def}[Thm]{Definiton}
\newtheorem{Prop}[Thm]{Proposition}
\newtheorem{Lem}[Thm]{Lemma}
\newtheorem{Cor}[Thm]{Corollary}

\numberwithin{equation}{section}
\begin{document}
\title{Finite pattern problems related to Engel expansion}
\author[hzau]{Chun-Yun Cao}
\ead{caochunyun@mail.hzau.edu.cn}
\author[hzau]{Yang Xiao}
\ead{1367379792@qq.com}

\address[hzau]{College of informatics, Huazhong Agricultural University, 430070 Wuhan, P.R.China.}

\begin{abstract}
Let $\mathcal{F}$ be a countable collection of functions $f$ defined on the integers with integer values, such that for every $f\in \mathcal{F}$, $f(n)\to +\infty$ as $n\to +\infty$. 
This paper primarily investigates the Hausdorff dimension of the set of points whose digit sequences of the Engel expansion are strictly increasing and contain any finite pattern of $\mathcal{F}$, demonstrating applications with representative examples.
\end{abstract}

\maketitle

\section{Introduction}

Finite-pattern problems represent a significant research area spanning multiple mathematical disciplines. 
In geometric measure theory, one frequently encounters questions such as: How "large" can a set be while avoiding certain patterns? or conversely,  how "small" can a set be while guaranteeing the presence of certain patterns? 
As an illustration, \cite{UY} shows that for any given dimension function $h$, and a family of functions $\mathcal{F}$ that
satisfy certain conditions, there is a closed set $E$ in $\mathds{R}^n$ , of $h$-Hausdorff
measure zero, such that for any finite set $\{f_1, \cdots, f_n\}\subset \mathcal{F}$, $E$ satisfies that
$\bigcap_{i=1}^nf_i(E)\neq \emptyset$. And there exists a closed set in $\mathds{R}$, of zero $h$-Hausdorff
measure but contains any polynomial pattern.
See \cite{FY,Fu,Y} and the references therein for more related results.

In the discrete case, the famour Roth-Szemerédi theorem \cite{Sz} provides the foundational result that any subset of  $\mathds{N}$ with  positive upper Banach density must contain arbitrary long arithmetic progressions.
The result of Green and Tao\cite{GT} shows that the converse fails, since the set of prime numbers contains arbitrarily long arithmetic progressions, yet has Banach density zero.
Here the upper Banach density of $A\subset \mathds{N}$ is defined as 
  $$\overline{d}_B(A)=\limsup_{n \to \infty}\frac{1}{n}\sup_{m\in \mathds{N}} \#\{i\in A: m\leq i< m+n\}.$$
  where \# denotes the cardinality of a set.
And for any integer $k\geq 3$ and real number $d\neq 0$, we call $A\subset \mathds{R}$ a $k$-term arithmetic progression with common difference $d$ if there exist some $t\in\mathds{R}$ such that $A=\{t+jd: j=0,\cdots,k-1\}$. 
And we call a set contains arbitrarily long arithmetic progressions if it contains a $k$-term arithmetic progression for any $k\geq 3$. 
Łaba and Pramanik \cite{LP} initiated the study on the relationship between Fourier decay and the existence of arithmetic progressions.The existence of arithmetic progressions is also closely related to  Assouad dimension.
For more results, please refer to \cite{DZ,FS,L,P,XJP} and the references therein. 

This paper primarily focuses on the finite-pattern problems related to Engel expansions.
The Engel expansion, named after the German mathematician Engel, is a unique representation of real numbers in the unit interval (0,1] as an infinite series of reciprocals of an increasing sequence of positive integers. 
Engel expansions can be reinterpreted via dynamical systems. 
For any $x\in(0,1]$, the Engel map $T:(0,1]\rightarrow (0,1]$ is defined by 
\begin{equation}\label{eq1}
    T(x):=\left(x-\frac{1}{n+1}\right)(n+1),\enspace \text{when}\enspace x\in\left(\frac{1}{n+1},\frac{1}{n}\right],\enspace n\geq 1.
\end{equation}
Then we define the integer sequence $\{d_n(x)\}_{n\geq 1}$ by
\begin{equation}\label{eq2}
  d_1(x)=\left\lfloor \frac{1}{x} \right\rfloor+1,\enspace d_n(x)=d_1(T^{n-1}(x)),\enspace n\geq 1,
\end{equation}
where $\left\lfloor x \right\rfloor $ denotes the integer part of $x$ and $T^n$ denotes the $n$th iteration 
of $T$ ($T^0$ is the identity transformation on $(0,1]$).
By \eqref{eq1} and \eqref{eq2}, any $x\in (0,1]$ can be uniquely represented as an infinite series expansion of the form
\begin{equation}\label{eq3}
  x=\frac{1}{d_1(x)}+\frac{1}{d_1(x)d_2(x)}+\cdots +\frac{1}{d_1(x)d_2(x) \cdots d_n(x)}+\cdots.
\end{equation}
We call the form \eqref{eq3} the Engel expansion of $x$ and denote it by $x=[d_1(x), d_2(x), \cdots, d_n(x), \cdots]$.
Furthermore, the positive integer sequence $\{d_n(x)\}_{n\geq1}$ is called the digit sequence of the Engel expansion of $x$.

Since the digit sequence of the Engel expansion is non-decreasing, it is natural to investigate whether the corresponding subset of natural numbers includes any arithmetic progressions.
Tian and Fang \cite{TF} showed that the set of $x\in (0, 1]$ for which $\{d_n(x)\}_{n\geq 1}$ contains arbitrarily long arithmetic progressions with arbitrary common difference is a Lebesgue null set, but has Hausdorff dimension one. 
Similar questions for other types of expansions have been extensively studied,
such as the continued fraction expansion \cite{TW}, the Lüroth expansion \cite{ZC1}. 
Note that the continued fraction system and the Lüroth system are both special $2$-decaying Gauss-like systems.
Zhang and Cao \cite{ZC2} extended the corresponding conclusions to a general $d$-decaying Gauss-like infinite iterated function systems, they showed that the Hausdorff dimension of the set of points whose digit sequences are strictly increasing and
of upper Banach density one is $\frac{1}{d}$.
Arithmetic progressions can be regarded as linear patterns. 
Tian, Wu and Lou \cite{TWL} investigated finite pattern problems related to Lüroth expansion.
More precisely, Let $\mathcal{F}$ be a set of countable many functions, where every function $f\in \mathcal{F}$ satifies 
\begin{equation}\label{f}
    f:\mathds{Z} \to \mathds{Z}  ~~\text{ and }~~  f(n)\to +\infty ~ \text{as} ~ n\to +\infty.
\end{equation}
A set $A$ of positive integers is said to contain any finite pattern of $\mathcal{F}$ if for any finitely many
functions $f_{n_1},f_{n_2},\cdots,f_{n_k}\in \mathcal{F}$, there exists a $n\in \mathds{N}$ such that
$$\{f_{n_1}(n),f_{n_2}(n),\cdots,f_{n_k}(n)\}\subset A.$$
Tian, Wu and Lou \cite{TWL} showed that the Hausdorff dimension of the set of the points whose Lüroth digit sequences are  strictly increasing and contains  any  finite pattern  of $\mathcal{F}$ is $1/2$. 
We aim to investigate the analogous problems in Engel expansions, with a focus on analyzing the size of the following set 
\[E_{\mathcal{F}}=\{x\in (0,1]:\{d_n(x)\}_{n\geq 1} ~\text{is strictly increasing and contains any finite pattern of}~\mathcal{F}\}.\]
\begin{Thm}\label{thF}
    For any set  $\mathcal{F}$ of countable many functions satisfying \eqref{f}, $\dim_HE_{\mathcal{F}}=1$.
\end{Thm}
As a consequence of this theorem, we obtain various finite pattern conclusions, with selected examples to be demonstrated.
The first example concerns digit sequences with positive upper Banach density.
\begin{Cor}\label{thB}
Let 
\begin{align*}
E_B=\{x\in (0,1]:&\{d_n(x)\}_{n\geq 1} ~\text{is strictly increasing and}~\\
 &\text{ is of upper Banach density one as  a  subset  of integers}\}.  
\end{align*}
Then $\dim_HE_{B}=1$.
\end{Cor}
The subsequent corollary follows immediately from either the Roth-Szemerédi theorem or, alternatively, may be deduced from Theorem \ref{thF}.
\begin{Cor}\label{thA}
Let 
\begin{align*}
E_A=\{x\in (0,1]:&\{d_n(x)\}_{n\geq 1} ~\text{is strictly increasing and contains arbitrarily long arithmetic progressions }\\
 &\text{with arbitrarily positive integer common  difference}\}.  
\end{align*}
Then 
$\dim_HE_{A}=1$.
\end{Cor}
Let $A,B\subset \mathds{N}$. 
By ``$A$ contains a translation of $B$'', we mean $\{x+n:x\in B\}\subset A$ for some integer $n$. A $k$-term arithmetic progression with common difference $d$ may be interpreted as a translation of the set $\{d,2d,\cdots,kd\}$. The framework actually yields a ``stronger'' version of Corollary \ref{thA}.
\begin{Cor}\label{thT}
Let 
\begin{align*}
E_T=\{x\in (0,1]:&\{d_n(x)\}_{n\geq 1} ~\text{is strictly increasing}\\
 &\text{and contains a translation of} ~B ~\text{for any finite subset}~ B ~\text{of}~ \mathds{N} \}.  
\end{align*}
Then 
$\dim_HE_{T}=1$.
\end{Cor}

Similarly, for $A,B\subset \mathds{N}$, we say that $A$ contains a scalar multiples of $B$ if $\{nx:x\in B\}\subset A$ for some integer $n$, and $A$ contains a power of $B$ if $\{x^n:x\in B\}\subset A$ for some integer $n$. Then we have

\begin{Cor}\label{thSP}
Let 
\begin{align*}
E_S=\{x\in (0,1]:&\{d_n(x)\}_{n\geq 1} ~\text{is strictly increasing}\\
 &\text{and contains a scalar multiples of} ~B ~\text{for any finite subset}~ B ~\text{of}~ \mathds{Z}^+ \}.  
 \end{align*}
 and 
 \begin{align*}
E_P=\{x\in (0,1]:&\{d_n(x)\}_{n\geq 1} ~\text{is strictly increasing}\\
 &\text{and contains a power of} ~B ~\text{for any finite subset}~ B ~\text{of}~ \{2,3,\cdots\} \}.  
\end{align*}
Then 
$\dim_HE_{S}=\dim_HE_{P}=1$.

Especially, let 
\begin{align*}
E_G=\{x\in (0,1]:&\{d_n(x)\}_{n\geq 1} ~\text{is strictly increasing and contains arbitrarily long geometric progressions }\\
 &\text{with arbitrarily positive integer common ratio}\}.  
\end{align*}
Then $\dim_HE_{G}=1$.
\end{Cor}
For more results on Engel expansions, please refer to \cite{FWW,FS,SLS,TF,ZS} and the references therein.

\section{Preliminaries}

Before we proceed with the proof of the theorems, in this section, we recall several definitions and 
basic properties of the Engel expansion, as well as
some important lemmas that will be needed later.
\begin{Def}(\cite{Gj})
  For a given positive integer $n$, a finite sequence $(d_1,d_2,\cdots,d_n)\in \mathds{N} ^n$ is said to be admissible 
  for Engel expansions if there exists $x\in (0,1]$ such that $d_j(x)=d_j$ for all $1\leq j \leq n$. An infinite 
  sequence $(d_1,d_2,\cdots,d_n,\cdots)\in \mathds{N} ^\mathds{N} $ is called an admissible sequence 
  if $(d_1,d_2,\cdots,d_n)$ is 
admissible for any $n \geq 1$.
\end{Def}

$\Sigma _n$ denotes the collecton of all the admissible sequences with length $n$ and $\Sigma $ denotes 
the collecton of all infinite admissible sequences. The following proposition gives a characterization of 
admissible sequences.

\begin{Prop}(\cite{Gj})
  A sequence $(d_1,d_2,\cdots,d_n,\cdots)\in \Sigma $ if and only if for all $n \geq 1$, 
  $$d_{n+1} \geq d_{n}\enspace and \enspace d_1 \geq 2.$$
\end{Prop}

Next we provide the definition of cylinders of Engel expansion, along with an examination of their structure and length.

\begin{Def}(\cite{Gj})
  Let $(d_1,d_2,\cdots,d_n)\in \Sigma _n$. We call 
  $$I_n(d_1,d_2,\cdots,d_n)=\{x \in (0,1]:d_1(x)=d_1,d_2(x)=d_2,\cdots,d_n(x)=d_n\}$$
  a cylinder of order n of the Engel expansions.
  
  It can be concluded that the digit sequences of the points in $I_n(d_1,d_2,\cdots,d_n)$ 
  begin with $(d_1,d_2,\cdots,d_n)$.
\end{Def}

\begin{Prop}\label{ZH}(\cite{Gj})
  Let $(d_1,d_2,\cdots,d_n)\in \Sigma _n$. The cylinder $I_n(d_1,d_2,\cdots,d_n)$ is a left-open, right-closed interval with the left endpoint 
  $$\frac{1}{d_1}+\frac{1}{d_1d_2}+\cdots+\frac{1}{d_1d_2 \cdots d_{n-1}}+\frac{1}{d_1d_2 \cdots d_{n-1} d_n},$$
  \\and the right endpoint 
  $$\frac{1}{d_1}+\frac{1}{d_1d_2}+\cdots+\frac{1}{d_1d_2 \cdots d_{n-1}}+\frac{1}{d_1d_2 \cdots d_{n-1}(d_n-1)}.$$
  The length of the interval $I_n(d_1,d_2,\cdots,d_n)$,  denoted by $|I_n(d_1,d_2,\cdots,d_n)|$,  is
  $$\frac{1}{d_1d_2\cdots d_{n-1}d_n(d_n-1)}.$$
\end{Prop}

The lemma that follows, known as Billingsley's Theorem, is a frequently employed method for 
estimating the lower bound of the Hausdorff dimension of fractal sets.
\begin{Lem}(\cite{B})\label{mass}
  Let $E\subset (0,1]$ be a Borel set and $\mu$ is a finite Borel measure on $[0,1]$. Suppose $\mu (E)>0$. If
  $$\liminf_{r\to 0^+}\frac{\log \mu(B(x,r))}{\log r}\geq s$$
for any $x\in E$, where $B(x,r)$ denotes the open ball with center at x and radius r, then $\dim_H E\geq s.$
\end{Lem}

At the end of this section, we introduce a lemma that aids in subsequent proof, delineating the relationship 
between the Hausdorff dimension of a set and that of its image under a quasi-Lipschitz mapping.
\begin{Def}\label{qua}(\cite{XLF})
Let $E,F\subset \mathds{R}^n$. A mapping $ f: E\to F $ is called quasi-Lipschitz if for every $\varepsilon>0$, there exists $\delta > 0$ such that 
\[\left|\frac{\log|f(x)-f(y)|}{\log |x-y|}-1\right|<\varepsilon ~\text{whenever}~ x,y\in E~\text{with}~ 0<|x-y|<\delta.\]
\end{Def}
\begin{Lem}\label{quasi}(\cite{XLF})
Every quasi-Lipschitz bijection preserves Hausdorff dimension.
\end{Lem}

\section{Proof of Theorem 1.2}
In this section, we provide the proof of Theorem \ref{thF}, 
namely, we demonstrate that the Hausdorff dimension of the set $E_{\mathcal{F}}$ is 1. 
Clearly, 1 is the upper bound of $\dim E_{\mathcal{F}}$, 
we only need to prove that the lower bound of $\dim E_{\mathcal{F}}$ is 1. 

\subsection{Subset of $E_{\mathcal{F}}$}

To begin with, we need to construct a strictly increasing sequence of positive integers 
that contains any finite pattern of $\mathcal{F}$. We can enumerate all non-empty finite subsets of $\mathcal{F}$ as 
$\mathcal{F}_1, \mathcal{F}_2, \cdots$, since it is countable. For $k\geq 1$, let
\begin{equation}\label{nk}
  n_k=\#\mathcal{F}_1+\#\mathcal{F}_2+\cdots+\#\mathcal{F}_k.
\end{equation}
According to \eqref{f}, for any integer $a>4$ , choose a sequence of positive integers $\{t_k\}_{k\geq 1}$ satisfying 
\begin{equation}\label{C2}
  \min_{f\in \mathcal{F}_{k}}f(t_{k})>2a^{n_k^2}~\text{for all}~k \geq 1.
\end{equation}
\begin{equation}\label{C1}
  \min_{f\in \mathcal{F}_{k+1}}f(t_{k+1})>\max_{f\in \mathcal{F}_{k}}f(t_{k})~\text{for all}~ k \geq 1.
\end{equation} 
By arranging the elements of set $\{f(t_k):f\in\mathcal{F}_k, k\geq 1\}$ in increasing order, 
we obtain the target sequence $\{b_m\}_{m\geq 1}$. 

Now, we introduce an auxiliary set to construct the target subset. For some integer $a>4$, define
\begin{align*}
  E_0=\{x\in(0,1]:~a^n\leq d_n(x)<2a^n ~\text{and}~d_n(x)\notin \{b_m:m\geq 1\}~\text{for all}~ n\geq 1\}.
\end{align*}
The following lemma related to the sequence $\{b_m\}_{m\geq 1}$ shows that the set $E_0$ is nonempty.
\begin{Lem}\label{bm}
  $\#\{m:b_m<2a^{n}\}<\sqrt{n}$  ~for all $n \geq 1$. 
\end{Lem}
\proof
For convenience, we extend the definition by letting $n_0=0$. Then for any $n \geq 1$, there exists some $k\in\mathbb{N}$ such that $n_k^2<n\leq n_{k+1}^2$. 
By \eqref{C1} and \eqref{C2}, 
\[b_m<2a^{n}\Rightarrow b_m\in \{f(t_i):f\in \mathcal{F},1\leq i\leq k\}.\]
Combining with \eqref{nk}, we conclude that 
\[\#\{m:b_m<2a^{n}\}\leq n_k<\sqrt{n}.\]
\endproof

By inserting the sequence $\{b_m\}_{m\geq 1}$ to the Engel digit sequence for every element in $E_0$, we obtain the target subset, 
which can be regarded as the image of $E_0$ under a certain mapping $\pi$.
For any $x\in E_0$, we combine sequences $\{b_m\}_{m\geq 1}$ and $\{d_n(x)\}_{n\geq 1}$ to form a new strictly increasing sequence $\{c_n\}_{n\geq 1}$. 
Clearly, there exists a unique $y\in (0,1]$ such that $d_n(y)=c_n$ for $n\geq 1$. 
This induces a mapping $\pi:x\mapsto y$ on $E_0$.
Define 
\[E_f=\{\pi(x):x\in E_0\}.\]
For any $y\in E_f$, its Engel digit sequence is strictly increasing and contain any finite pattern of $\mathcal{F}$ since $\{b_m\}_{m\geq 1} \subset \{d_n(y)\}_{n\geq 1}$. 
Therefore, we have $E_f \subset E_{\mathcal{F}}$.

\subsection{Hausdorff dimension of $E_{\mathcal{F}}$}
We begin by proving that the mapping $\pi$ defined above is quasi-Lipschitz, which implies $\dim_H E_f =\dim_H E_0$ by Lemma \ref{quasi}.

\begin{Lem}\label{pi}
The mapping $\pi:E_0\to E_f$ is quasi-Lipschitz.
\end{Lem}
\proof
For any $x_1,x_2\in E_0$ with $x_1\neq x_2$, 
there exists $n\in \mathbb{N}$ such that $x_1,x_2\in I_n(d_1,d_2,\cdots,d_n)$ and $d_{n+1}(x_1)\neq d_{n+1}(x_2)$. 
Assume that $d_{n+1}(x_1)>d_{n+1}(x_2)$, we have the cylinder of order $n+1$ containing $x_1$ is on the left of the 
the cylinder of order $n+1$ containing $x_2$. Furthermore, We discuss the relationships between cylinders through the following computations:
\[
  \frac{\sum_{k=d_{n+1}(x_2)+1}^{\infty}|I_{n+2}(d_1,d_2,\cdots,d_n,d_{n+1}(x_2),k)|}{|I_{n+1}(d_1,d_2,\cdots,d_n,d_{n+1}(x_2))|}
  =\frac{d_{n+1}(x_2)-1}{d_{n+1}(x_2)}>\frac{1}{2a^2},
\]
since $d_n(x_2)\geq a^n>4$.
Based on the above equation, we can obtain:
\begin{equation}\label{juli}
  \frac{|I_{n+1}(d_1,d_2,\cdots,d_n,d_{n+1}(x_2))|}{2a^2}\leq |x_1-x_2|\leq |I_n(d_1,d_2,\cdots,d_n)|
\end{equation}

Let $y_1=\pi(x_1),\enspace y_2=\pi(x_2)$. Assume that $b_l<d_{n+1}(x_2)<b_{l+1}$, according to $d_{n+1}(x_2)<2a^{n+1}$ and 
Lemma \ref{bm}, we have
\begin{equation}\label{l}
  l<\sqrt{n+1}.
\end{equation}
It is evident that the first $n+l$ digits in the Engel expansion of $y_1$ and $y_2$ are the same. Combining sequences 
${d_1,d_2,\cdots,d_n}$ and $b_1,b_2,\cdots,b_l$ in increasing order to form sequence $c_1,c_2,\cdots,c_{n+l}$, we have 
$$y_1,y_2\in I_{n+l}(c_1,c_2,\cdots,c_{n+l}),$$
and
$$d_{n+l+1}(y_2)=d_{n+1}(x_2)<\min(d_{n+1}(x_1),b_{l+1})=d_{n+l+1}(y_1).$$
Similar to the previous calculations, we can obtain analogous conclusions:
\begin{equation}\label{juli2}
  \frac{|I_{n+l+1}(c_1,c_2,\cdots,c_{n+l},d_{n+l+1}(y_2))|}{2a^2}\leq |y_1-y_2|\leq |I_{n+l}(c_1,c_2,\cdots,c_{n+l})|.
\end{equation}

By proposition \ref{ZH} and the definition of $E_0$, we obtain that
\begin{equation}\label{1}
  -\log|I_{n+1}(d_1,d_2,\cdots,d_{n+1}(x_2))|\sim -\log|I_n(d_1,d_2,\cdots,d_n)|\sim \frac{1}{2}n^2\log a
\end{equation}
as $n\to \infty$, where $f(n)\sim g(n)$ as $n \to \infty$ means that $\lim_{n\to \infty}f(n)/g(n)=1$.
From \eqref{l}, we also have
\begin{equation}\label{2}
  -\log|I_{c+l+1}(c_1,c_2,\cdots,c_{n+l},d_{n+l+1}(y_2))|\sim -\log|I_{n+l}(c_1,c_2,\cdots,c_{n+l})|\sim \frac{1}{2}n^2\log a
\end{equation}
as $n\to \infty$.

Finally, Combining \eqref{juli}, \eqref{juli2}, \eqref{1} and \eqref{2}, We have demonstrated that
$$\frac{\log|\pi(x_1)-\pi(x_2)|}{\log|x_1-x_2|}\to 1$$
uniformly as $|x_1-x_2|\to 0$. 
\endproof

Now we will determine the Hausdorff dimension of $E_0$. 
Consequently, combining Lemma \ref{quasi} with Lemma \ref{pi} yields
\[\dim_H E_\mathcal{F}\geq \dim_H E_f =\dim_H E_0.\]

\begin{Lem}\label{e0}
  For any integer $a>4$ , $\dim_HE_0=1$.
\end{Lem}

\begin{proof}
Let $a>4$ be an integer, for any $n \geq1$, set
\[D_n=\{(d_1,d_2,\cdots,d_n)\in \mathds{N}^n: ~a^k\leq d_k(x)<2a^k~\text{and}~
  d_k(x)\notin \{b_m:m\geq 1\}~\text{for all}~1\leq k\leq n\}.\]
By Lemma \ref{bm}, we have
\begin{equation}
  \#D_n>\prod_{k=1}^n(a^k-\sqrt{k}).
\end{equation}
For each $n\geq1$ and $(d_1,d_2,\cdots,d_n)\in D_n$, we call $I_n(d_1,d_2,\cdots,d_n)$ an admissible cylinder of order $n$ with respect to $E_0$.
Then
$$E_0=\bigcap_{n=1}^{\infty}\bigcup_{(d_1,d_2,\cdots,d_n)\in D_n}\text{cl}\enspace I_n(d_1,d_2,\cdots,d_n).$$
where “cl” denotes closure.

Now we define a set function $\mu$ on all admissible cylinders with respect to $E_0$ by setting
$$\mu(I_n(d_1,d_2,\cdots,d_n))=\frac{1}{\#D_n}.$$
By the Carath$\acute{\text{e}}$odory extension theorem, $\mu$ can be extended to be a  Borel probability measure supported
on $E_0$. We will estimate the Hausdorff dimension of $E_0$ based on the Billingsley's Theorem.
  
For each $x\in E_0$, there exists a sequence $\{d_1,d_2,\cdots,d_n,\cdots\}$ such that for each $n\geq1$, 
$x\in I_n(d_1,d_2,\cdots,d_n)$ and $(d_1,d_2,\cdots,d_n)\in D_n$. 
For any $0<r<\frac{1}{8^a}$, there exists $n\geq 1$ such that 
\[|I_n(d_1,d_2,\cdots,d_n)|>r\geq |I_{n+1}(d_1,d_2,\cdots,d_n,d_{n+1})|.\]
For any adjacent admissible cylinders of order $n$ with respect to $E_0$, we have 
\[
  \frac{|I_n(d_1,d_2,\cdots,d_n+1)|+|I_n(d_1,d_2,\cdots,d_n+2)|}{|I_n(d_1,d_2,\cdots,d_n)|}
  =\frac{2d_n-2}{d_n+2}>1,
\]
since $d_n\geq a^n>4$. Then $B(x,r)$ can intersect at most four admissible cylinders of order $n$ with respect to $E_0$ and at least one admissible cylinder of order $n+1$ with respect to $E_0$. Therefore
\begin{align*}
  \liminf_{r\to 0}\frac{\log \mu(B(x,r))}{\log r}
  &\geq \liminf_{n\to \infty}\frac{\log4\mu(I_n(d_1,d_2,\cdots,d_n))}{\log|I_{n+1}(d_1,d_2,\cdots,d_{n+1})|}\\
  &=\liminf_{n\to \infty}\frac{\log\#D_n-\log 4}{\sum_{k=1}^{n+1}\log d_k+\log(d_{n+1}-1)}\\
  &\geq\liminf_{n\to \infty}\frac{\sum_{k=1}^{n}\log(a^k-\sqrt{k})}{\sum_{k=1}^{n+1}\log(2a^k)+\log(2a^{n+1})}\\
  &=1.
\end{align*}
By Lemma \ref{mass}, we obtain that $\dim_HE_0\geq1$. Therefore  $\dim_HE_0=1$.
\end{proof}

\section{Applications of Theorem\ref{thF}}
In this section, we derive the corollaries stated in the introduction through a careful choice of $\mathcal{F}$ and certain finite subsets of $\mathcal{F}$.

At first, Take 
\[\mathcal{F}=\{kn+\ell: k\in \mathds{Z}^+, \ell\in \mathds{N}\}.\]
For any $x\in E_{\mathcal{F}}$, $m\in\mathds{Z}^+$ and any $\{f_{i_1},f_{i_2},\cdots,f_{i_m}\}\subset \mathcal{F}$, there exists $n\in \mathds{N}$ such that  
\[\{f_{i_1}(n),f_{i_2}(n),\cdots,f_{i_m}(n)\}\subset \{d_n(x)\}_{n\geq 1}.\]

(1)$\{n+1,n+2,\cdots,n+m\}\subset \mathcal{F}$ for any $m\in \mathds{Z}^+$.

Suppose there exists $n_m\in \mathds{N}$ such that  
\[\{n_m+1,n_m+2,\cdots,n_m+n\}\subset \{d_n(x)\}_{n\geq 1}.\]
Then 
\begin{align*}
\overline{d}(\{d_n(x)\}_{n\geq 1})&=
\limsup_{m \to \infty}\frac{1}{m}\sup_{n\in \mathds{N}} \#\{i\in \{d_n(x)\}_{n\geq 1}: n\leq i< n+m\})\\
&\geq  \limsup_{m \to \infty}\frac{1}{m} \#\{i\in \{d_n(x)\}_{n\geq 1}: n_m\leq i< n_m+m\})\\
&\geq  \limsup_{m \to \infty}\frac{m-1}{m}=1.
\end{align*}
Therefore, $E_\mathcal{F}\subset E_B$. It can be concluded that $\dim_H E_B=1$ by Theorem \ref{thF}.

(2)$\{n+b:b\in B\}\subset \mathcal{F}$ for any finite subset $B\subset \mathds{N}$.

There exists $n_B\in \mathds{N}$ such that $\{n_B+b:b\in B\}\subset \{d_n(x)\}_{n\geq 1}$. 
Thus, $E_\mathcal{F}\subset E_T$. As a result $\dim_H E_T=1$.

Especially, for any integer $k\geq 3$ and $d\geq 1$, $B=\{d,2d,\cdots,kd\}$ is a finite subset of $\mathds{N}$.
The translation of $B$ is precisely a $k$-term arithmetic progression with common difference $d$.
This leads to $E_\mathcal{F}\subset E_A$, and $\dim_H E_A=1$.

(3)$\{bn:b\in B\}\subset \mathcal{F}$ for any finite subset $B\subset \mathds{Z}^+$.

There exists $\widetilde{n}_B\in \mathds{N}$ such that $\{b \widetilde{n}_B:b\in B\}\subset \{d_n(x)\}_{n\geq 1}$. 
It follows that $E_\mathcal{F}\subset E_S$. As a result $\dim_H E_S=1$.

Especially, for any integer $k\geq 3$ and $q\geq 2$, $B=\{q,q^2,\cdots,q^k\}$ is a finite subset of $\mathds{Z}^+$.
The  scalar multiples  of $B$, $\{\widetilde{n}_B q,\widetilde{n}_B q^2,\cdots,  \widetilde{n}_B q^k\}$, 
is precisely a $k$-term geometric  progression with common ratio $q$.
This leads to $E_\mathcal{F}\subset E_G$, and $\dim_H E_G=1$.

Now, Take 
\[\mathcal{F}=\{k^n: k\in \mathds{Z}^+,k\geq 2\}.\]
For any finite subset $B\subset \{2,3,4,\cdots\}$, $\{b^n:b\in B\}$ is a finte subset of $\mathcal{F}$.
For any $x\in E_{\mathcal{F}}$, there exists $\widehat{n}_B\in\mathds{N}$ 
such that $\{b^{\widehat{n}_B}:b\in B\}\subset \{d_n(x)\}_{n\geq 1}$. 
That is to say $\{d_n(x)\}_{n\geq 1}$ contains a power of $B$.
Hence,  $E_\mathcal{F}\subset E_P$.  From this we have $\dim_H E_P=1$ by Theorem \ref{thF}.

\end{document}